\documentclass[12pt]{amsart}

\usepackage[utf8]{inputenc}
\usepackage[american]{babel}
\usepackage{amsmath,amssymb,amsthm}
\usepackage{enumerate}
\usepackage{dsfont}
\usepackage[left=2.5cm,right=2.5cm,top=3cm,bottom=3cm]{geometry}
\usepackage{csquotes}
\usepackage[style=alphabetic,sorting=anyt,backend=bibtex,doi=false,url=false,maxbibnames=10,isbn=false]{biblatex}
\usepackage[colorlinks,linkcolor={black},citecolor={black},urlcolor={black}]{hyperref}
\usepackage{cleveref}

\newcommand{\IC}{\mathbb C}
\newcommand{\IR}{\mathbb R}
\newcommand{\IZ}{\mathbb Z}

\renewcommand{\H}{\mathcal H}
\newcommand{\J}{\mathcal J}
\renewcommand{\L}{\mathcal L}
\newcommand{\M}{\mathcal M}
\newcommand{\N}{\mathcal N}

\renewcommand{\phi}{\varphi}

\newcommand{\id}{\mathrm{id}}

\newcommand{\abs}[1]{\lvert#1\rvert}
\newcommand{\norm}[1]{\lVert#1\rVert}

\newtheorem{theorem}{Theorem}[section]
\newtheorem{proposition}[theorem]{Proposition}
\newtheorem{corollary}[theorem]{Corollary}
\newtheorem{lemma}[theorem]{Lemma}

\theoremstyle{remark}
\newtheorem{remark}[theorem]{Remark}
\newtheorem{example}[theorem]{Example}

\title[Symmetric Christensen--Evans theorem]{Christensen--Evans theorem and extensions of GNS-symmetric quantum Markov semigroups}
\author{Melchior Wirth}
\address{Institute of Science and Technology Austria (ISTA), Am Campus 1, 3400 Klosterneuburg, \textsf{Austria}}
\email{melchior.wirth@ist.ac.at}

\begin{document}

\begin{abstract}
In this note we prove a refined version of the Christensen--Evans theorem for generators of uniformly continuous GNS-symmetric quantum Markov semigroups. We use this result to show the existence of GNS-symmetric extensions of GNS-symmetric quantum Markov semigroups. In particular, this implies that the generators of GNS-symmetric quantum Markov semigroups on finite-dimensional von Neumann algebra can be written in the form specified by Alicki's theorem.
\end{abstract}

\maketitle

\section{Introduction}

Quantum Markov semigroups were originally introduced in mathematical physics to model the time evolution of certain open quantum systems and have since developed into an object of broad mathematical interest in pure and applied disciplines. Particular research effort has been spent on understanding the structure of their generators. For quantum Markov semigroups acting on matrix algebras, the results by Lindblad \cite{Lin76} and Gorini--Kossakowski--Sudarshan \cite{GKS76} give a full description of the generators, which are now commonly called Lindblad operators. In the broader context of uniformly continuous quantum Markov semigroups on von Neumann algebras, Christensen and Evans \cite{CE79} showed that the generator is of the form
\begin{equation*}
\L(x)=k^\ast x+xk-\Phi(x)
\end{equation*}
with a completely positive map $\Phi$.

In many situations however, the suitable category is not that of von Neumann algebras, but that of $W^\ast$-probability spaces, that is, von Neumann algebras with a fixed normal faithful state. In this situation, one is often interested in quantum Markov semigroups that satisfy a symmetry condition with respect to the reference state, such as GNS- or KMS-symmetry. Again, these conditions have a physical motivation in the description of open quantum systems whose environment is in thermal equilibrium. The structure of the generators of GNS- and KMS-symmetric quantum Markov semigroups on matrix algebras (and more generally the full algebra of bounded operators on a Hilbert space) has been intensively studied, too. In particular, Alicki's theorem \cite{Ali76} gives a full description of the generators of a GNS-symmetric quantum Markov semigroups on matrix algebras in similar form to the Lindblad--Gorini--Kossakowski--Sudarshan theorem. However, an analog of the Christensen--Evans theorem for GNS-symmetric quantum Markov semigroups on arbitrary von Neumann algebras seems to be missing so far.

Recent years have seen a renewed interest in the structure of the generators of quantum Markov semigroups in the study of noncommutative functional inequalities with applications in quantum information. In particular, Alicki's theorem has been instrumental in the definition of a noncommutative analog of the $2$-Wasserstein distance by Carlen and Maas \cite{CM17a}; and the differential calculus developed there has found applications beyond \cite{Bar17,GR21}. But to use Alicki's theorem for a GNS-symmetric quantum Markov semigroup on a finite-dimensional von Neumann algebra, it has to be assumed that it has a GNS-symmetric extensions to the full matrix algebra. While this can be shown for many examples of interest, the question whether this is always the case has been left open.

In this article we show that the generator of a uniformly continuous GNS-symmetric quantum Markov semigroup has a Christensen--Evans form in which the operator $k$ and the completely positive map $\Phi$ are compatible with GNS-symmetry, namely $k$ is in the centralizer of the state (and positive) and $\Phi$ is GNS-symmetric (\Cref{thm:CE_symmetric}).
 
Our proof relies on the theory of von Neumann bimodules and correspondences. We show that if a uniformly continuous quantum Markov semigroup is GNS-symmetric, the associated GNS bimodule carries some extra structure (\Cref{prop:Tomita_bimodule}). This can be seen as a first step in extending the first-order differential calculus developed by Cipriani and Sauvageot \cite{CS03} to quantum Markov semigroups that are GNS-symmetric with respect to a non-tracial state (or weight). A crucial step in the proof of the original Christensen--Evans theorem is that a derivation associated with the quantum Markov semigroup is inner. Likewise it will be central in our proof that in the case of a GNS-symmetric quantum Markov semigroup, the vector representing this derivation can be chosen compatible with the extra structure mentioned above.
 
As an application we show that a uniformly continuous GNS-symmetric quantum Markov semigroup on a von Neumann subalgebra can be extended to a GNS-symmetric quantum Markov semigroup on the full von Neumann algebra provided the subalgebra is the range of a normal conditional expectation (\Cref{thm:extension}). In particular, Alicki's theorem can be applied to any GNS-symmetric quantum Markov semigroup on a unital $\ast$-subalgebra of a matrix algebra (\Cref{cor:Alicki}). In the appendix we give a short proof of the fact (originally due to Kadison and Ringrose \cite{KR71a,KR71b}) that every approximately finite-dimensional von Neumann algebra with separable predual is amenable, which might be of independent interest.

\subsection*{Acknowledgments} The author wants to thank Jan Maas for raising the question of extending GNS-symmetric quantum Markov semigroups and Michael Skeide for introducing him to his work on von Neumann modules. The author acknowledges financial support from the European Research Council (ERC) under
the European Union’s Horizon 2020 research and innovation programme (grant agreement No 716117).

\section{GNS-symmetric maps}

In this section we recall some definitions and prove two elementary facts about GNS-symmetric completely positive maps.

Let $\M$ be a von Neumann algebra and $\phi$ a normal faithful state on $\M$. We write $(\pi_\phi,L_2(\M,\phi),\Omega_\phi)$ for the cyclic representation associated with $\phi$ and $\sigma^\phi$ for the modular group. Let $\M_{\sigma^\phi}$ denote the set of analytic elements for $\sigma^\phi$, that is, the set of all $x\in \M$ for which $t\mapsto \sigma^\phi_t(x)$ has an analytic continuation to $\IC$. This continuation is denoted by $z\mapsto \sigma^\phi_z(x)$. Recall that $\M_{\sigma^\phi}$ is a weak$^\ast$ dense $\ast$-subalgebra of $\M$.

A bounded linear operator $P\colon \M\to\M$ is called \emph{GNS-symmetric} with respect to $\phi$ (or simply \emph{$\phi$-symmetric}) if
\begin{equation*}
\phi(P(x)^\ast y)=\phi(x^\ast P(y))
\end{equation*}
for all $x,y\in \M$.

The main object of interest of this article are GNS-symmetric quantum Markov semigroups, that is, weakly continuous semigroups of GNS symmetric unital completely positive maps on $\M$. It is customary to additionally assume that the individual maps are normal and commute with the modular group, but we will show that these two requirements are automatically satisfied.

For normality this is certainly known to experts, but as we could not find a precise reference, we give the short proof for the convenience of the reader.
\begin{lemma}\label{lem:GNS_normal}
Every GNS-symmetric unital completely positive map is normal.
\end{lemma}
\begin{proof}
Let $P$ be a $\phi$-symmetric unital completely positive map. By \cite[Proposition III.2.2.2]{Bla06} it suffices to show that $P$ is strongly continuous on the unit ball of $\M$. By \cite[Proposition III.5.3]{Tak02}, on the unit ball the strong topology coincides with the topology induced by the norm
\begin{equation*}
\norm{\cdot}_\phi\colon\M\to[0,\infty),\,\norm{x}_\phi=\phi(x^\ast x)^{1/2}.
\end{equation*}
By the Kadison--Schwarz inequality and $\phi$-symmetry,
\begin{equation*}
\norm{P(x)}_\phi=\phi(P(x)^\ast P(x))^{1/2}\leq \phi(P(x^\ast x))^{1/2}=\phi(x^\ast x)^{1/2}=\norm{x}_\phi.
\end{equation*}
Hence $P$ is continuous with respect to $\norm{\cdot}_\phi$.
\end{proof}

That GNS-symmetric unital completely positive maps commute with the modular group is also known for matrix algebras \cite[Lemma 2]{Ali76}, \cite[Lemma 2.5]{CM17a}. Our proof in the infinite-dimensional case follows the same strategy, just that some care has to be taken of the domains.
\begin{proposition}\label{prop:mod_group_commute}
If $P$ is a $\phi$-symmetric unital completely positive map, then $P$ commutes with the modular group $\sigma^\phi$.
\end{proposition}
\begin{proof}
As shown in the proof of the previous lemma, the map
\begin{equation*}
\M\Omega_\phi\to L_2(\M,\phi),\,x\Omega_\phi\mapsto P(x)\Omega_\phi
\end{equation*}
extends to a bounded linear operator $\tilde P$ on $L_2(\M,\phi)$, and $\phi$-symmetry of $P$ implies that $\tilde P$ is symmetric.

We consider the left Hilbert algebra $\mathfrak{A}=\M\Omega_\phi$ as in \cite[Example VI.1.3]{Tak03} with modular operator $\Delta$, modular conjugation $J$ and $S=J\Delta^{1/2}$. Since $P$ is hermitian, $\tilde P$ commutes with $S$.

Let
\begin{equation*}
\mathfrak{A}_0=\left\{\xi \in \bigcap_{n\in\IZ}D(\Delta^n):\Delta^n\xi\in\mathfrak{A}\text{ for all }n\in\IZ\right\}
\end{equation*}
be the associated  Tomita algebra.

In the following we will freely use the relations between $S$, $J$ and $\Delta$ as detailed in \cite[Lemma VI.1.5]{Tak03}. For $\xi,\eta\in \mathfrak A_0$ we have
\begin{align*}
\langle\tilde P\Delta \xi,\eta\rangle&=\langle \Delta \xi,\tilde P\eta\rangle\\
&=\langle J\tilde P\eta,J\Delta \xi\rangle\\
&=\langle \Delta^{1/2}S\tilde P\eta,S\Delta^{1/2}\xi\rangle\\
&=\langle S\eta,\tilde P\Delta^{1/2}S\Delta^{1/2}\xi\rangle\\
&=\langle S\eta,S\tilde P\xi\rangle\\
&=\langle JS\tilde P\xi,JS\eta\rangle\\
&=\langle \Delta^{1/2}\tilde P \xi,\Delta^{1/2}\eta\rangle\\
&=\langle \tilde P\xi,\Delta \eta\rangle.
\end{align*}
Since $\mathfrak{A}_0$ is a core for $\Delta$ by \cite[Theorem VI.2.2 (ii)]{Tak03}, it follows that $\tilde P\xi\in D(\Delta)$ and $\Delta \tilde P\xi=\tilde P\Delta\xi$ for all $\xi\in\mathfrak{A}_0$. Again since $\mathfrak{A}_0$ is a core, this implies $\Delta\tilde P=\tilde P\Delta$ and thus $\Delta^{it}\tilde P=\tilde P\Delta^{it}$ for all $t\in\IR$.

Therefore
\begin{equation*}
\sigma^\phi_t(P(x))\Omega_\phi=\Delta^{it}(P(x)\Omega_\phi)=\Delta^{it}\tilde P(x\Omega_\phi)=\tilde P\Delta^{it}(x\Omega_\phi)=P(\sigma^\phi_t(x))\Omega_\phi.
\end{equation*}
As $\Omega_\phi$ is a separating vector for $\M$, this proves the claim.
\end{proof}

\begin{remark}
Usually the extension of $P$ from $\M$ to the GNS Hilbert space $H$ is not defined via the embedding $x\mapsto x\Omega_\phi$ as in this proof, but via the symmetric embedding $x\mapsto \Delta^{1/4}(x\Omega_\phi)$. However, commutation with the modular group as proven here implies that both extensions actually coincide.
\end{remark}

\section{Von Neumann modules and correspondences}

Throughout this article we use the language of von Neumann modules and correspondences, which we briefly recall now. As a detailed reference for von Neumann modules we refer to \cite{Ske01}, while the necessary results on correspondences can be found in \cite[Section IX.3]{Tak03}.

Von Neumann modules are noncommutative analogs of $L^\infty$ sections of a vector bundle, while correspondences are analogs of $L^2$ sections. In the presence of a normal faithful state on the underlying von Neumann algebra, both carry the same information and one can go back and forth between these two notions (see \Cref{lem:ident_left_bounded}), which will be useful later on.

Let $A$ be a unital $C^\ast$-algebra. A \emph{pre-$C^\ast$ $A$-module} is a right $A$-module $F$ with a sesquilinear map $(\cdot\mid\cdot)\colon F\times F\to A$ such that
\begin{itemize}
\item $(\xi|\eta a)=(\xi|\eta)a$ for all $\xi,\eta\in F$, $a\in A$,
\item $(\xi|\xi)\geq 0$ for all $\xi\in F$,
\item $(\xi|\xi)=0$ implies $\xi=0$.
\end{itemize}
A \emph{$C^\ast$ $A$-module} is a pre-$C^\ast$ $A$-module $F$ that is complete in the norm
\begin{equation*}
\norm{\xi}_F=\norm{(\xi|\xi)}^{1/2}.
\end{equation*}
A $C^\ast$-module over $\IC$ is nothing but a Hilbert space.

A bounded linear operator $T$ on a $C^\ast$-module $F$ is called \emph{adjointable} if there exists a bounded linear operator $T^\ast$ on $F$ such that
\begin{equation*}
(T\xi|\eta)=(\xi|T^\ast\eta)
\end{equation*}
for all $\xi,\eta\in F$. Note that all adjointable operators are right module maps, that is, $T(\xi a)=(T\xi)a$ for all $a\in A$, $\xi\in F$.

Let $A$, $B$ and $C$ be unital $C^\ast$-algebras. A \emph{$C^\ast$ $A$-$B$-module} is a $C^\ast$-bimodule together with an action of $A$ by adjointable operators. In particular, a $C^\ast$ $\IC$-$A$-bimodule is the same as a $C^\ast$ $A$-module and a $C^\ast$ $A$-$\IC$-bimodule the same as a representation of $A$ on a Hilbert space. In the case $A=B$ we simply speak of $C^\ast$ $A$-bimodules.

The \emph{tensor product} $F\bar\odot G$ of a $C^\ast$ $A$-$B$-module $F$ and a $C^\ast$ $B$-$C$-module is the $C^\ast$ $A$-$C$-module obtained from the algebraic tensor product $F\otimes G$ after separation and completion with respect to the $C$-valued inner product given by
\begin{equation*}
(\xi\otimes\eta|\xi'\otimes\eta')=(\eta|(\xi|\xi')\eta')
\end{equation*}
and the actions given by
\begin{equation*}
a(\xi\otimes\eta)=a\xi\otimes \eta,\,(\xi\otimes\eta)c=\xi\otimes\eta c.
\end{equation*}
If $A$ is a $C^\ast$-algebra of bounded operators on $H$ and $F$ is a $C^\ast$ $A$-module, we can embed $F$ into $B(H,F\bar\odot H)$ by the action 
\begin{equation*}
H\to F\bar\odot H,\,\zeta\mapsto \xi\otimes\zeta
\end{equation*}
for $\xi\in F$. If we refer to the strong topology on a $C^\ast$-module in the following, we always mean the strong topology in this embedding, where $H$ is clear from the context (usually the standard form of a von Neumann algebra).

Let $\M$ be a von Neumann algebra on $H$. A \emph{von Neumann $\M$-module} is a $C^\ast$ $\M$-module $F$ that is strongly closed in $B(H,F\bar\odot H)$. The adjointable operators on a von Neumann $\M$-module form a von Neumann algebra $\L_\M(F)$.

If $\N$ is another von Neumann algebra on $K$, then a $C^\ast$ $\M$-$\N$-module is a \emph{von Neumann $\M$-$\N$-module} if it is a von Neumann $\N$-module and the left action of $\M$ on $F\bar\odot K$ is normal.

Assume that $\psi$ is a normal faithful state on $\N$. If $H=L_2(\N,\psi)$, the Hilbert space $F\bar\odot H$ carries not only a left action of $\N$, but also a right action given by
\begin{equation*}
(\xi\otimes\eta)x=\xi\otimes\eta x
\end{equation*}
for $\xi\in F$, $\eta\in L_2(\N,\psi)$ and $x\in \M$. This makes $F\bar\odot L_2(\N,\psi)$ into an \emph{$\M$-$\N$-correspondence} in the sense of Connes, that is, a Hilbert space with commuting representations of $\M$ and $\N^\circ$, the opposite $\ast$-algebra of $\N$.

Conversely, one can recover $F$ from $F\bar\odot L_2(\N,\psi)$. To prove this fact, we need some more terminology. Let $\H$ be an $\M$-$\M$-correspondence. A vector $\xi\in \H$ is called \emph{left-bounded} if there exists a constant $C(\xi)>0$ such that
\begin{equation*}
\norm{\xi y}_\H\leq C(\xi)\norm{\Omega_\psi y}_2
\end{equation*}
for all $y\in \N$. The set of all left-bounded vectors in denoted by $L_\infty(\H_\N,\psi)$. By definition, for every $\xi\in L_\infty(\H_\N,\psi)$ the map
\begin{equation*}
\Omega_\psi\N\to \H,\,\Omega_\psi x\mapsto \xi x
\end{equation*}
extends to a bounded right module map $L(\xi)$ from $L_2(\N,\psi)$ to $\H$. In particular, if $\xi,\eta\in L_\infty(\H_\N,\psi)$, then $L(\xi)^\ast L(\eta)$ commutes with the right action of $\N$ on $L_2(\N,\psi)$, which implies $L(\xi)^\ast L(\eta)\in \N$. Conversely, if $T\colon L_2(\N,\psi)\to \H$ is a bounded right module map, then it is not hard to see that $T\Omega_\psi$ is left-bounded and $T=L(T\Omega_\psi)$.

\begin{proposition}\label{lem:ident_left_bounded}
The map
\begin{equation*}
F\to F\bar\odot L_2(\N,\psi),\,\xi\mapsto \xi\otimes\Omega_\psi
\end{equation*}
is a bijection onto $L_\infty((F\bar\odot L_2(\N,\psi))_\M,\psi)$ and
\begin{equation*}
L(\xi\otimes\Omega_\psi)^\ast L(\zeta\otimes\Omega_\psi)=(\xi|\zeta)
\end{equation*}
for all $\xi,\zeta\in F$.

Conversely, if $\H$ is an $\M$-$\N$-correspondence, then $L_\infty(\H_\N,\psi)$ endowed with the actions
\begin{align*}
\M\times L_\infty(\H_\N,\psi)\to L_\infty(\H_\M,\phi),\,(x,\xi)\mapsto x\xi,\\
L_\infty(\H_\N,\psi)\times \N\to L_\infty(\H_\N,\psi),\,(\xi,y)\mapsto L(\xi)y\Omega_\psi
\end{align*}
and the $\N$-valued inner product $(\xi|\eta)=L(\xi)^\ast L(\eta)$ is a von Neumann $\M$-$\N$-module, and the map
\begin{equation*}
L_\infty(\H_\N,\psi)\otimes \Omega_\psi\N\to \H,\,\xi\otimes \Omega_\psi x\mapsto \xi x
\end{equation*}
extends to a unitary from $L_\infty(\H_\N,\psi)\bar\odot L_2(\N,\psi)$ to $\H$.
\end{proposition}
\begin{proof}
If $y\in \N$ and $\xi\in F$, then
\begin{equation*}
\norm{(\xi\otimes\Omega_\psi)y}^2=\langle\Omega_\psi y,(\xi|\xi)\Omega_\psi y\rangle_2\leq \norm{(\xi|\xi)}\norm{\Omega_\psi y}_2^2.
\end{equation*}
Thus $\xi\otimes\Omega_\phi$ is left-bounded. The equality $L(\xi\otimes\Omega_\psi)^\ast L(\zeta\otimes\Omega_\psi)=(\xi|\zeta)$ for $\xi,\zeta\in F$ is straightforward.

To prove that the map $\xi\mapsto \xi\otimes\Omega_\psi$ is a surjection onto $L_\infty((F\bar\odot L_2(\N,\psi))_\N,\psi)$, let $\eta\in L_\infty((F\bar\odot L_2(\N,\psi))_\N,\psi)$ and define the map
\begin{equation*}
\Phi\colon F\to \N,\,\xi\mapsto L(\eta)^\ast L(\xi\otimes\Omega_\psi)
\end{equation*}

We have
\begin{align*}
\Phi(\xi x)(\Omega_\psi y)&=L(\eta)^\ast (\xi x\otimes \Omega_\psi y)\\
&=L(\eta)^\ast(\xi\otimes x\Omega_\psi y)\\
&=L(\eta)^\ast(\xi\otimes \Omega_\psi \sigma^\psi_{i/2}(x)y)\\
&=L(\eta)^\ast L(\xi\otimes \Omega_\psi)\Omega_\psi \sigma^\psi_{i/2}(x) y\\
&=\Phi(\xi)(x\Omega_\psi y)
\end{align*}
for all $x\in \N_{\sigma^\psi}$, $y\in\N$ and $\xi\in F$. By approximation, this implies that $\Phi$ is a right module map.

Since von Neumann modules are self-dual \cite[Theorem 3.2.11]{Ske01}, there exists $\zeta\in F$ such that
\begin{equation*}
L(\eta)^\ast L(\xi\otimes\Omega_\phi)=(\zeta|\xi)=L(\zeta\otimes\Omega_\phi)^\ast L(\xi\otimes\Omega_\phi)
\end{equation*}
for all $\xi\in F$. Since the linear hull of elements of the form $L(\xi\otimes\Omega_\psi)x$ with $\xi\in F$ and $y\in \N$ is dense in $F\bar\odot L_2(\N,\psi)$, we have $L(\eta)=L(\zeta\otimes\Omega_\psi)$, which implies $\eta=\zeta\otimes\Omega_\psi$ by applying both operators to $\Omega_\psi$.

For the converse, first note that whenever $\xi\in L_\infty(\H_\N,\psi)$ and $x,y\in \N$, then
\begin{equation*}
\norm{(L(\xi)x\Omega_\psi)y}_\H=\norm{L(\xi)(x\Omega_\psi y)}_\H\leq \norm{L(\xi)}\norm{x}\norm{\Omega_\psi y}_2,
\end{equation*}
so that $L(\xi)x\Omega_\psi$ is left-bounded and thus the right action is well-defined. It is not difficult to check that this makes $L_\infty(\H_\N,\psi)$ a $C^\ast$ $\M$-$\N$-module and that $L_\infty(\H_\N,\psi)\bar\odot L_2(\N,\psi)\cong \H$ under the map above.

With this identification, the canonical embedding 
\begin{equation*}
L_\infty(\H_\N,\psi)\hookrightarrow B(L_2(\N,\psi),L_\infty(\H_\N,\psi)\bar\odot L_2(\N,\psi))
\end{equation*}
is just the map $\xi\mapsto L(\xi)$. As discussed above, $\{L(\xi)\mid \xi\in L_\infty(\H_\M,\phi)\}$ is the set of all bounded right module maps from $L_2(\N,\psi)$ to $\H$, which is clearly strongly closed in $B(L_2(\N,\psi),\H)$. Normality of the actions is easy to check, so that $L_\infty(\H_\N,\psi)$ is a von Neumann $\M$-bimodule.
\end{proof}
\begin{remark}
As mentioned above, left-bounded vectors in $\mathcal H$ are in one-to-one correspondence with right module maps from $L_2(\N,\psi)$ to $\mathcal H$. That these form a von Neumann $\M$-$\N$-module has already been observed in \cite[Section 1.6]{Sch02}.
\end{remark}

If $\H_1$ is an $\M$-$\N$-correspondence and $\H_2$ is an $\N$-$\mathcal R$-correspondences, their \emph{relative tensor product} or \emph{Connes fusion tensor product} $\H_1\otimes_\psi\H_2$ is the $\M$-$\mathcal R$-correspondence obtained from $L_\infty((\H_1)_\N,\psi)\otimes \H_2$ after separation and completion with respect to
\begin{equation*}
B(\xi_1\otimes\eta_1,\xi_2\otimes\eta_2)=\langle \eta_1,L(\xi_1)^\ast L(\xi_2)\eta_2\rangle_{\H_2}.
\end{equation*}
The following identification is a straightforward consequence of the previous lemma.

\begin{lemma}\label{lem:rel_tensor_product}
If $F_1$, $F_2$ are von Neumann $\M$-$\M$-bimodules, then the map
\begin{align*}
F_1\otimes F_2\otimes\Omega_\phi&\to (F_1\bar\odot L_2(\M,\phi))\otimes_\phi(F_2\bar\odot L_2(\M,\phi)),\\
\xi\otimes\eta\otimes \Omega_\phi&\mapsto (\xi\otimes \Omega_\phi)\otimes_\phi (\eta\otimes \Omega_\phi)
\end{align*}
extends to a unitary from $(F_1\bar\odot F_2)\bar\odot L_2(\M,\phi)$ onto $(F_1\bar\odot L_2(\M,\phi))\otimes_\phi(F_2\bar\odot L_2(\M,\phi))$.
\end{lemma}

Besides left-bounded vectors, there is the dual notion of right-bounded vectors. An element $\xi\in \H$ is called \emph{right-bounded} if there exists a constant $C(\xi)>0$ such that
\begin{equation*}
\norm{x\xi}_\H\leq C(\xi)\norm{x\Omega_\phi}_2.
\end{equation*}
The set of all right-bounded vectors is denoted by $L_\infty(_\M\H,\phi)$. By definition, for every $\xi\in L_\infty(_\M\H,\phi)$ the map
\begin{equation*}
\M\Omega_\phi\to\H,\,x\Omega_\phi\mapsto x\xi
\end{equation*}
extends to a bounded left module map $R(\xi)$ from $L_2(\M,\phi)$ to $\H$. Analogously to the case of left-bounded vectors, this implies $R(\xi)^\ast R(\eta)\in \M'$ for $\xi,\eta\in L_\infty(_\M\H,\phi)$.

Let $\H_1$, $\H_2$ be $\M$-$\M$-correspondences. By \cite[Proposition IX.3.15]{Tak03}, the semi-inner product form
\begin{equation*}
B'\colon\H_1\otimes L_\infty(_\M\H_2,\phi)\to \IC,\,B'( \xi_1\otimes \eta_1,\xi_2\otimes\eta_2)=\langle \xi_1\cdot JR(\eta_1)^\ast R(\eta_2)J,\xi_2\rangle
\end{equation*}
coincides with $B$ on $L_\infty((\H_1)_\M,\phi)\otimes L_\infty(_\M\H_2,\phi)$. Thus the $\M$-$\M$-correspondence obtained from $\H_1\otimes L_\infty(_\M\H_2,\phi)$ after separation and completion with respect to $B'$ is canonically isomorphic to $\H_1\otimes_\phi\H_2$.

\section{Christensen--Evans theorem for GNS-symmetric quantum Markov semigroups}

In this section we prove the first main result of this article, a refined version of the Christensen--Evans theorem \cite[Theorem 3.1]{CE79} for the generators of uniformly continuous GNS--symmetric quantum Markov semigroups (\Cref{thm:CE_symmetric}). The Christensen--Evans theorem asserts that the generator $\L$ of a uniformly continuous quantum Markov semigroup on a von Neumann algebra $\M$ is of the form
\begin{equation*}
\L(x)=k^\ast x+x k-\Phi(x)
\end{equation*}
with $k\in \M$ and a completely positive map $\Phi\colon\M\to \M$ such that $\Phi(1)=k+k^\ast$.

We will show that if the given quantum Markov semigroup is $\phi$-symmetric, then $k$ can be chosen positive and in the centralizer of $\phi$ and $\Phi$ can be chosen $\phi$-symmetric. 

Let us first recall the definition of GNS-symmetric quantum Markov semigroups. Let $\M$ be a von Neumann algebra and $\phi$ a normal faithful state on $\M$. A \emph{quantum Markov semigroup} is a family $(P_t)_{t\geq 0}$ of normal unital completely positive maps on $\M$ such that
\begin{enumerate}[(i)]
\item $P_0=I$,
\item $P_s P_t=P_{s+t}$ for $s,t\geq 0$,
\item $P_t(x)\to x$ weakly as $t\searrow 0$ for every $x\in \M$.
\end{enumerate}
A quantum Markov semigroup $(P_t)$ is called \emph{GNS-symmetric} if the individual maps $P_t$ are GNS-symmetric for all $t\geq 0$. Note that in this case normality of the maps $P_t$ is automatic by \Cref{lem:GNS_normal} and furthermore every $P_t$ commutes with the modular group $\sigma^\phi$ by \Cref{prop:mod_group_commute}.

If the convergence in (iii) is in norm for every $x\in \M$, then the quantum Markov semigroup is called \emph{uniformly continuous}. In this case, the generator $\L$ defined by
\begin{equation*}
\L(x)=\lim_{t\searrow 0}\frac 1 t(x-P_t(x))
\end{equation*}
is a normal bounded operator on $\M$. It is known that a normal bounded linear map $\L$ on $\M$ generates a quantum Markov semigroup if and only if $\L(1)=0$ and $\L$ is conditionally negative definite, that is,
\begin{equation*}
\sum_{j,k=1}^n x_j^\ast \L(a_j^\ast a_k)x_k\leq 0
\end{equation*}
whenever $a_1,\dots,a_n,x_1,\dots,x_n\in \M$ such that $\sum_j a_j x_j=0$. Moreover, it is easy to see that the quantum Markov semigroup generated by $\L$ is GNS-symmetric if and only if $\L$ is GNS-symmetric. The associated \emph{carré du champ} is the $\M$-valued bilinear map
\begin{equation*}
\Gamma\colon\M\times\M\to\M,\,\Gamma(x,y)=\frac 1 2(\L(x)^\ast y+x^\ast\L(y)-\L(x^\ast y).
\end{equation*}
We write $\Gamma(x)$ for $\Gamma(x,x)$.

Using the conditional negative definiteness of the generator, one can associate a von Neumann $\M$-bimodule with a quantum Markov semigroup. This construction is known as the \emph{GNS construction} for quantum Markov semigroups. Let us briefly sketch it.

On $\M\otimes \M$ define an $\M$-valued bilinear form $B$ by
\begin{equation*}
B(a\otimes x,b\otimes y)=-\frac  1 2 x^\ast \L(a^\ast b)y.
\end{equation*}
Since $\L$ is conditionally negative definite, $B$ is positive semi-definite on the subspace $F_0=\{\sum_j a_j\otimes x_j\mid \sum_j a_j x_j=0\}$. Let $F_0^\prime$ be the $C^\ast$ $\M$-bimodule obtained from $F_0$ after separation and completion with respect to $B$ and write $(\cdot|\cdot)$ for its $\M$-valued inner product. The strong closure $\overline{F_0^\prime}^s$ of $F_0^\prime$ inside $B(L_2(\M,\phi),F_0^\prime\bar\odot L_2(\M,\phi))$ is a von Neumann $\M$-bimodule.

Define a derivation $\delta$ on $\M$ with values in $\overline{F_0^\prime}^s$ by
\begin{equation*}
\delta(x)=x\otimes 1-1\otimes x.
\end{equation*}
A direct calculation shows
\begin{equation*}
(\delta(x)|\delta(y))=\frac 1 2 (\L(x)^\ast y+x^\ast \L(y)-\L(x^\ast y))=\Gamma(x,y).
\end{equation*}
The von Neumann $\M$-bimodule $F$ associated with $(P_t)$ is the strong closure of $\{x\delta(y)\mid x,y\in \M\}$ inside $\overline{F_0^\prime}^s$. 

The von Neumann $\M$-bimodule $F$ and the derivation $\delta$ are essentially uniquely determined by $(P_t)$. To show this, we will make use of the fact that von Neumann bimodules are dual spaces. More precisely, according to \cite[Proposition 3.8]{Pas73}, $F$ is the dual space of $\M_\ast\otimes_\pi \bar F/\ker \iota$, where $\bar F$ is the conjugate space of $F$ and
\begin{equation*}
\iota\colon \M_\ast\otimes_\pi\bar F\to F^\ast,\,\omega\otimes \zeta\mapsto \omega((\zeta|\,\cdot\,)),
\end{equation*}
with the duality given by
\begin{equation*}
\M_\ast\otimes_\pi\bar F/\ker\iota\times F\to\IC,\,(\psi+\iota,\xi)\mapsto \iota(\psi)(\xi).
\end{equation*}
In fact, the predual of a von Neumann module is unique according to \cite[Theorem 2.6]{Sch02}, which justifies to denote it by $F_\ast$.

\begin{proposition}\label{prop:F_unique}
If $F'$ is a von Neumann $\M$-bimodule and $\delta'\colon \M\to F'$ is a derivation such that $(\delta'(x)|\delta'(y))=\Gamma(x,y)$ for all $x,y\in\M$ and $\{\delta'(x)y\mid x,y\in\M\}$ is strongly dense in $F'$, then there exists a unique weak$^\ast$ continuous bimodule isomorphism $\alpha\colon F\to F'$ that satisfies 
\begin{equation*}
(\alpha(\xi)|\alpha(\zeta))=(\xi|\zeta)
\end{equation*}
for all $\xi,\zeta\in F$ and $\alpha\circ\delta=\delta'$.
\end{proposition}
\begin{proof}
Uniqueness of the map $\alpha$ is obvious from the density assumptions. Let us show existence. By construction, $\{\delta(x)y\mid x,y\in\M\}$ is strongly dense in $F$. Define 
\begin{equation*}
U\colon \{\delta(x)y\otimes\Omega_\phi\mid x,y\in\M\}\to F'\bar\odot L_2(\M,\phi),\,U(\delta(x)y\otimes\Omega_\phi)=\delta'(x)y\otimes\Omega_\phi.
\end{equation*}
Since
\begin{align*}
\langle \delta'(a)x\otimes\Omega_\phi,\delta'(b)y\otimes\Omega_\phi\rangle&=\langle \Omega_\phi,(\delta'(a)x|\delta'(b)y)\Omega_\phi\rangle\\
&=\langle \Omega_\phi,x^\ast(\delta'(a)|\delta'(b))y\Omega_\phi\rangle\\
&=\langle\Omega_\phi,x^\ast\Gamma(a,b)y\Omega_\phi\rangle\\
&=\langle \Omega_\phi,x^\ast(\delta(a)|\delta(b))y\Omega_\phi\rangle\\
&=\langle \delta(a)x\otimes\Omega_\phi,\delta(b)y\otimes\Omega_\phi\rangle
\end{align*}
for all $a,b,x,y\in \M$, the map $U$ extends to a unitary from $F\bar\odot L_2(\M,\phi)$ to $F'\bar\odot L_2(\M,\phi)$.

Taking into account that $x\delta(y)=\delta(xy)-\delta(x)y$, it is easy to check that $U$ is a bimodule map. Hence $U$ maps left-bounded vectors to left-bounded vectors and $L(U\eta)=UL(\eta)$ for $\eta\in L_\infty((F\bar\odot L_2(\M,\phi))_\M,\phi)$.

By \Cref{lem:ident_left_bounded} there exists a unique map $\alpha\colon F\to F'$ such that $\alpha(\xi)\otimes\Omega_\phi=U(\xi\otimes\Omega_\phi)$ for all $\xi\in F$. In particular, $\alpha$ is a bijective bimodule map. Moreover,
\begin{equation*}
\alpha(\delta(x))\otimes\Omega_\phi=U(\delta(x)\otimes\Omega_\phi)=\delta'(x)\otimes\Omega_\phi,
\end{equation*}
which implies $\alpha\circ\delta=\delta'$, and
\begin{align*}
(\alpha(\xi)|\alpha(\zeta))&=L(\alpha(\xi)\otimes\Omega_\phi)^\ast L(\alpha(\zeta)\otimes\Omega_\phi)\\
&=L(U(\xi\otimes\Omega_\phi))^\ast L(U(\zeta\otimes\Omega_\phi))\\
&=L(\xi\otimes\Omega_\phi)^\ast U^\ast U L(\zeta\otimes\Omega_\phi)\\
&=(\xi|\zeta).
\end{align*}
It remains to show that $\alpha$ is weak$^\ast$ continuous. Recall the definition of $F_\ast$ given above. Let $\iota_F$ denote the canonical map from $\M_\ast\otimes_\pi \bar F$ to $F^\ast$ and write $\iota_{F'}$ for the analog map for $F'$.

Since
\begin{equation*}
\sum_k \norm{\omega_k}\norm{\alpha^{-1}(\xi_k)}=\sum_k\norm{\omega_k}\norm{\xi_k},
\end{equation*}
the map 
\begin{equation*}
\M_\ast\otimes \bar F'\to \M_\ast\otimes \bar F,\,\omega\otimes\xi\mapsto \omega\otimes \alpha^{-1}(\xi)
\end{equation*}
extends to an isometry $\beta\colon \M_\ast\otimes_\pi \bar F'\to \M_\ast\otimes_\pi \bar F$. Moreover, as
\begin{equation*}
\iota_F (\omega\otimes \alpha^{-1}(\xi))(\zeta)=\omega((\alpha^{-1}(\xi)|\zeta))=\omega((\xi|\alpha(\zeta)))=\iota_{F'}(\omega\otimes\xi)(\alpha(\zeta)),
\end{equation*}
the map $\beta$ maps $\ker \iota_{F'}$ to $\ker\iota_F$ and the quotient map $\alpha_\ast\colon F'_\ast\to F_\ast$ is a predual of $\alpha$. Hence $\alpha$ is weak$^\ast$ continuous.
\end{proof}

\begin{remark}\label{rmk:deriv_CS}
Assume that $\phi$ is a trace. By \cite{CS03}, there exists an $\M$-$\M$-correspondence $\H$ and a closable derivation $\partial\colon\M\to\H$ such that $\phi(\Gamma(x))=\norm{\delta(x)}^2$. One can show that 
\begin{equation*}
\norm{\partial(x)y}^2=\phi(\Gamma(x)yy^\ast).
\end{equation*}
This implies
\begin{equation*}
\norm{\partial(x)y}^2\leq \norm{\Gamma(x)}\phi(yy^\ast).
\end{equation*}
Hence $\partial(x)\in L_\infty(\H_\M,\phi)$ and $L(\partial(x))^\ast L(\partial(x))=\Gamma(x)$. Note that since $\phi$ is a trace, $L(\eta)x\Omega_\phi=\eta x$. If we view $L_\infty(\H_\M,\phi)$ as a von Neumann $\M$-bimodule as described in \Cref{lem:ident_left_bounded}, $\partial$ is thus a derivation with values in $L_\infty(\H_\M,\phi)$. Therefore $F$ can be realized as a subbimodule of $L_\infty(\H_\M,\phi)$ and $\delta$ as a corestriction of $\partial$.
\end{remark}

\begin{example}[Semigroups of Herz--Schur multipliers]\label{ex:groups}
Let $G$ be a discrete group. A conditionally negative definite length function on $G$ is a map $\ell\colon G\to[0,\infty)$ such that $\ell(e)=0$, $\ell(g^{-1})=\ell(g)$ and
\begin{equation*}
\sum_{g,h\in G}\bar\alpha_g\alpha_h\ell(g^{-1}h)\leq 0
\end{equation*}
for every $\alpha\in C_c(G)$ with $\sum_{g\in G}\alpha_g=0$. If $\ell$ is a negative definite length function, the map $P_t$ on $\IC[G]$ defined by $P_t \lambda_g=e^{-t\ell(g)}\lambda_g$ extends to a normal unital completely positive map on the group von Neumann algebra $L(G)$, still denoted by $P_t$. The operators $P_t$ form a $\tau$-symmetric quantum Markov semigroup, where $\tau$ is the trace on $L(G)$ given by $\tau(x)=\langle \delta_e,x\delta_e\rangle$.

For every conditionally negative definite length function $\ell$ there exists a (real) Hilbert space $H$, an orthogonal representation $\pi$ of $G$ on $H$ and a map $b\colon G\to H$ satisfying the $1$-cocycle condition
\begin{equation*}
b(gh)=b(g)+\pi(g)b(h)
\end{equation*}
such that $\norm{b(g)-b(h)}^2=\ell(g^{-1}h)$.

In this case, the $L(G)$-$L(G)$-correspondence $\H$ from \Cref{rmk:deriv_CS} is a subcorrespondence of $\H= H_\IC\otimes\ell_2(G)$, where $H_\IC$ is an $L(G)$-$\IC$-correspondence through the action of $L(G)$ on $H_\IC$ induced by $\pi$. The derivation $\partial$ is given by $\partial(\lambda_g)=\delta_g\otimes b(g)$ \cite[Section 10.2]{CS03}.

The semigroup $(P_t)$ is uniformly continuous if and only if $\ell$ is bounded. In this case, according to \Cref{rmk:deriv_CS} the von Neumann $L(G)$-bimodule $F$ can be realized as a subbimodule of $H_\IC\otimes L(G)$ with the standard inner product and right action and the left action given by $\lambda_g(\eta\otimes\lambda_h)=\pi(g)\eta\otimes\lambda_{gh}$. The derivation $\delta$ is then given by $\delta(\lambda_g)=b(g)\otimes \lambda_g$.
\end{example}

So far, the construction of $F$ and $\delta$ works for arbitrary quantum Markov semigroups and does not use GNS-symmetry. We will next show that GNS-symmetry gives rise to some extra structure on $F$.

\begin{proposition}\label{prop:Tomita_bimodule}
Let $(P_t)$ be a uniformly continuous GNS-symmetric quantum Markov semigroup. There exist a unique semigroup $(V_t)$ of weak$^\ast$ continuous isometries of $F$ and a unique anti-unitary involution $\J$ on $F\bar\odot L_2(\M,\phi)$ such that
\begin{enumerate}[(a)]
\item $V_t(\delta(x))=\delta(\sigma^\phi_t(x))$ for all $x\in \M$, $t\in \IR$,
\item $V_t(x\xi y)=\sigma^\phi_t(x)(V_t \xi)\sigma^\phi_t(y)$ for all $x,y\in \M$, $\xi\in F$, $t\in\IR$,
\item $\J(\delta(x)\otimes\Omega_\phi)=\delta(\sigma^\phi_{i/2}(x)^\ast)\otimes\Omega_\phi$ for all $x\in \M_{\sigma^\phi}$,
\item $\J(x\eta y)=y^\ast (\J\eta)x^\ast$ for all $x,y\in \M$, $\eta\in F\bar\odot L_2(\M,\phi)$.
\end{enumerate}
Moreover, for every $t\in \IR$ the map
\begin{equation*}
F\otimes \Omega_\phi\to F\bar\odot L_2(\M,\phi),\,\xi\otimes \Omega_\phi\mapsto V_t\xi\otimes\Omega_\phi
\end{equation*}
extends to a unitary $U_t$ on $F\bar\odot L_2(\M,\phi)$ and
\begin{enumerate}[(a)]
\setcounter{enumi}{4}
\item $\J\circ U_t=U_t\circ \J$ for all $t\in \IR$.
\end{enumerate}
\end{proposition}
\begin{proof}
Since the linear hull of $\{x\delta(y)\mid x,y\in \M\}$ is strongly dense in $F$ and $F\otimes \Omega_\phi$ is dense in $F\bar\odot L_2(\M,\phi)$, uniqueness is clear.

Let us start with the existence of $V_t$. We will use that $\L$ commutes with modular group by \Cref{prop:mod_group_commute}. Let $a_j,x_j,b_k,y_k\in \M$ with $\sum_j a_jx_j=\sum_k b_k y_k=0$ and $t\in \IR$. As
\begin{align*}
&\quad\;\left\langle\sum_j \sigma^\phi_t(a_j)\otimes\sigma^\phi_t(x_j)\otimes\Omega_\phi,\sum_k\sigma^\phi_t(b_k)\otimes\sigma^\phi_t(y_k)\otimes\Omega_\phi\right\rangle\\
&=-\frac 1 2\sum_{j,k}\phi(\sigma^\phi_t(x_j^\ast)\L(\sigma^\phi_t(a_j^\ast b_k))\sigma^\phi_t(y_k))\\
&=-\frac 1 2\sum_{j,k}\phi(\sigma^\phi_t(x_j^\ast\L(a_j^\ast b_k)y_k))\\
&=\left\langle\sum_j a_j\otimes x_j\otimes\Omega_\phi,\sum_k b_k\otimes y_k\otimes\Omega_\phi\right\rangle,
\end{align*}
the map
\begin{equation*}
\sum_j a_j\otimes x_j\otimes\Omega_\phi\mapsto\sum_j \sigma^\phi_t(a_j)\otimes \sigma^\phi_t(x_j)\otimes\Omega_\phi
\end{equation*}
extends to a unitary $U_t$ on $F\bar\odot L_2(\M,\phi)$ and
\begin{equation*}
U_t(x\eta y)=\sigma^\phi_t(x)(U_t\eta)\sigma^\phi_t(y)
\end{equation*}
for all $x,y\in \M$, $\eta\in F\bar\odot L_2(\M,\phi)$, $t\in\IR$. In particular, if $\eta\in L_\infty(_\M\H,\phi)$, then $U_t\eta\in L_\infty(\H_\M,\phi)$ and $L(U_t \eta)=L(\eta)\Delta_\phi^{-it}$.

Thus $U_t$ restricts to a map from $L_\infty(\H_\M,\phi)$ to itself with $\norm{L(U_t \eta)}= \norm{L(\eta)}$. The existence of an isometry $V_t$ on $F$ with properties (a) and (b) as well as the relation $U_t(\xi\otimes\Omega_\phi)=V_t\xi\otimes\Omega_\phi$ follows from the identifications from \Cref{lem:ident_left_bounded}. Note that we have
\begin{equation*}
(V_t\xi|V_t\zeta)=L(U_t(\xi\otimes\Omega_\phi))^\ast L(U_t(\zeta\otimes\Omega_\phi))=\Delta_\phi^{it}L(\xi\otimes\Omega_\phi)^\ast L(\zeta\otimes\Omega_\phi)\Delta_\phi^{-it}=\sigma^\phi_t((\xi|\zeta)).
\end{equation*}

It just remains to show that $V_t$ is weak$^\ast$ continuous. The argument is as the one for the weak$^\ast$ continuity of $\alpha$ in the proof of \Cref{prop:F_unique}.

Since
\begin{align*}
\sum_k \norm{\omega_k\circ\sigma^\phi_t}\norm{V_{-t}\zeta_k}=\sum_k\norm{\omega_k}\norm{\zeta_k},
\end{align*}
the map
\begin{equation*}
\M_\ast\otimes \bar F\to\M_\ast\otimes\bar F,\,\omega\otimes \zeta\mapsto \omega\circ\sigma^\phi_t\otimes V_{-t}\zeta
\end{equation*}
extends to an isometry $W_t$ on $M_\ast\otimes_\pi\bar F$. Moreover, as
\begin{equation*}
\iota(\omega\circ\sigma^\phi_t\otimes V_{-t}\zeta)(\xi)=\omega(\sigma^\phi_t(V_{-t}\zeta|\xi))=\omega((\zeta|V_t\xi))=\iota(\omega\otimes\zeta)(V_t\xi),
\end{equation*}
the map $W_t$ leaves $\ker\iota$ invariant and the quotient map $(V_t)_\ast\colon F_\ast\to F_\ast$ is a predual of $V_t$. Hence $V_t$ is weak$^\ast$ continuous.

Let us now come to the existence of $\J$. We will use $\phi$-symmetry of $\L$ and in particular \Cref{prop:mod_group_commute}.

Let $a_j,x_j,b_k,y_k\in \M_{\sigma^\phi}$ with $\sum_j a_j x_j=\sum_k b_k y_k=0$.
Since
\begin{align*}
&\quad\;\left\langle\sum_k \sigma^\phi_{i/2}(y_k)^\ast\otimes \sigma^\phi_{i/2}(b_k)^\ast\otimes\Omega_\phi,\sum_j \sigma^\phi_{i/2}(x_j)^\ast\otimes \sigma^\phi_{i/2}(a_j)^\ast\otimes\Omega_\phi\right\rangle\\
&=-\frac 1 2\sum_{j,k}\phi(\sigma^\phi_{i/2}(b_k)\L(\sigma^\phi_{i/2}(y_k)\sigma^\phi_{i/2}(x_j)^\ast)\sigma^\phi_{i/2}(a_j)^\ast)\\
&=-\frac 1 2\sum_{j,k}\phi(\L(\sigma^\phi_{i/2}(y_k)\sigma^\phi_{i/2}(x_j)^\ast)\sigma^\phi_{-i/2}(a_j^\ast b_k))\\
&=-\frac 1 2\sum_{j,k}\phi(\sigma^\phi_{i/2}(y_k)\sigma^\phi_{i/2}(x_j)^\ast\L(\sigma^\phi_{-i/2}(a_j^\ast b_k)))\\
&=-\frac 1 2\sum_{j,k}\phi(\sigma^\phi_{-i/2}(x_j^\ast)\sigma^\phi_{-i/2}(\L(a_j^\ast b_k))\sigma^\phi_{-i/2}(y_k))\\
&=-\frac 1 2\sum_{j,k}\phi(x_j^\ast\L(a_j^\ast b_k) y_k)\\
&=\left\langle \sum_j a_j\otimes x_j\otimes\Omega_\phi,b_k\otimes y_k\otimes\Omega_\phi\right\rangle,
\end{align*}
the map
\begin{equation*}
\sum_j a_j\otimes x_j\otimes\Omega_\phi\mapsto -\sum_j \sigma^\phi_{i/2}(x_j)^\ast\otimes\sigma^\phi(a_j)^\ast\otimes\Omega_\phi
\end{equation*}
extends to an anti-unitary operator $\J$ on $F\bar\odot L_2(\M,\phi)$. Properties (c)--(e) as well as $\J^2=\J$ are easy to check.
\end{proof}

\begin{remark}\label{rmk:J_left_bounded}
If $\eta\in L_\infty((F\bar\odot L_2(\M,\phi))_\M,\phi)$ and $x\in\M$, then
\begin{equation*}
\norm{x\J\eta}=\norm{\J(\eta x^\ast)}=\norm{\eta x^\ast}\leq\norm{L(\eta)}\norm{\Omega_\phi x^\ast}=\norm{L(\eta)}\norm{x\Omega_\phi}.
\end{equation*}
Thus $\J$ maps left-bounded vectors to right-bounded vectors (and vice versa).
\end{remark}

\begin{remark}
If we write $L$ and $R$ for the left and right action of $\M$ and $\M^\circ$ on $F\bar\odot L_2(\M,\phi)$, respectively, then property (b) can be expressed as $L(\sigma^{\phi}_t(x))=U_t L(x) U_{-t}$ and $R(\sigma^\phi_t(x))=U_t R(x) U_{-t}$. In other words, $(L,(U_t)_{t\in\IR})$ and $(R,(U_t)_{t\in\IR})$ are covariant representations of $(\M,\IR,\sigma^{\phi})$ and $(\M^\circ,\IR,\sigma^\phi)$, respectively.
\end{remark}

\begin{remark}\label{rmk:diff_calc}
If $\phi$ is a trace, the map 
\begin{equation*}
\partial\colon \M\to F\bar\odot L_2(\M,\phi),\,x\mapsto\delta(x)\otimes\Omega_\phi
\end{equation*}
satisfies $\partial(xy)=x\partial(y)+\partial(x)y$ and $\partial(x^\ast)=\J\partial(x)$. Thus the structure obtained in the previous proposition can be seen as a non-tracial version of Cipriani's and Sauvageot's first-order differential calculus \cite{CS03} in the uniformly continuous case. Note that if $\phi$ is a trace, then $V_t=\mathrm{id}$ for all $t\in\IR$, so that the group $(V_t)$ is a feature of the non-tracial case.
\end{remark}

We record one observation from the proof of \Cref{prop:Tomita_bimodule}, as it will be needed later.

\begin{lemma}\label{lem:mod_group_inner_prod}
If $\xi,\zeta\in F$ and $t\in\IR$, then 
\begin{equation*}
\sigma^\phi_t((\xi|\zeta))=(V_t\xi|V_t\zeta).
\end{equation*}
In particular, if $\xi,\zeta$ are invariant under $(V_t)$ and $x\in \M_{\sigma^\phi}$, then $(\xi|x\zeta)\in \M_{\sigma^\phi}$ and
\begin{equation*}
\sigma^\phi_z((\xi|x\zeta))=(\xi|\sigma^\phi_z(x)\zeta)
\end{equation*}
for all $z\in\IC$.
\end{lemma}

Since $(P_t)$ is uniformly bounded, the derivation $\delta\colon \M\to F$ is bounded. It is a crucial step in the proof of the Christensen--Evans theorem to show that this implies that $\delta$ is inner \cite[Theorem 2.1]{CE79}. For $\zeta\in F$ we will write 
\begin{equation*}
\delta_\zeta\colon \M\to F,\,x\mapsto x\zeta-\zeta x.
\end{equation*}
We will show in two steps that GNS-symmetry of $(P_t)$ implies that $\delta$ is of the form $\delta_\xi$ with a vector $\xi\in F$ that is invariant under the maps $V_t$ and $\J$ from \Cref{prop:Tomita_bimodule}. This will be instrumental in showing that the completely positive map in the Christensen--Evans theorem can be chosen GNS-symmetric.

\begin{lemma}\label{lem:V_t_invariant}
There exists $\xi'\in F$ such that $V_t\xi'=\xi'$ for all $t\in\IR$ and $\delta_{\xi'}=\delta$.
\end{lemma}
\begin{proof}
By \cite[Theorem 2.1]{CE79} there exists $\xi^{\prime\prime}\in F$ such that $\delta_{\xi^{\prime\prime}}=\delta$. Let
\begin{equation*}
C=\{\zeta\in F\mid \delta_\zeta=\delta,\,\norm\zeta\leq\norm{\xi^{\prime\prime}}\}.
\end{equation*}
This set is bounded and weak$^\ast$ closed, hence weak$^\ast$ compact by the Banach-Alaoglu theorem. Moreover it is convex.

If $\zeta\in C$, then
\begin{equation*}
\delta_{V_t\zeta}(x)=x V_t\zeta-(V_t\zeta)x=V_t(\sigma^\phi_{-t}(x)\zeta-\zeta\sigma^\phi_{-t}(x))=V_t(\delta(\sigma^\phi_{-t}))=\delta(x)
\end{equation*}
for all $x\in \M$ and $t\in\IR$ by \Cref{prop:Tomita_bimodule} (a). Moreover, $\norm{V_t\zeta}=\norm{\zeta}\leq\norm{\xi^{\prime\prime}}$. Thus $V_t\zeta\in C$.

Since each $V_t$ is weak$^\ast$ continuous and $V_sV_t=V_{s+t}=V_tV_s$, we can apply the Markov--Kakutani fixed-point theorem to obtain $\xi'\in C$ with $V_t\xi'=\xi'$ for all $t\in\IR$.
\end{proof}
\begin{remark}
In the light of \Cref{lem:mod_group_inner_prod}, $V_t\xi'=\xi'$ implies
\begin{equation*}
\sigma^\phi_t(\xi'|x\xi')=(\xi'|\sigma^\phi_t(x)\xi').
\end{equation*}
\end{remark}

The following result is a standard consequence of interpolation theory for noncommutative $L_p$ spaces.
\begin{lemma}\label{lem:generator_bounded}
For all $x\in\M$ one has
\begin{equation*}
\phi(\L(x)^\ast x)\leq \norm{\L}\phi(x^\ast x).
\end{equation*}
\end{lemma}
\begin{proof}
The symmetric embedding of $\M$ into $L_2(\M,\phi)$ is given by 
\begin{equation*}
j_\phi\colon\M\to L_2(\M,\phi),\,x\mapsto \Delta_\phi^{1/4}(x\Omega_\phi),
\end{equation*}
and the symmetric embedding of $\M$ into $\M_\ast$ is $i_\phi=j_\phi^\ast J j_\phi$, where $J$ is viewed as a linear map from $L_2(\M,\phi)$ to $L_2(\M,\phi)^\ast$, identified via the Riesz representation theorem with the conjugate space of $L_2(\M,\phi)$.

By GNS-symmetry, the operator
\begin{equation*}
\L_\ast\colon i_\phi(\M)\to \M_\ast,\,\L_\ast(i_\phi(x))=i_\phi(\L(x))
\end{equation*}
satisfies $(\L_\ast)^\ast\subset \L$. Hence $\L_\ast$ is bounded with $\norm{\L_\ast}=\norm{\L}$.

By interpolation theory for noncommutative $L_p$ spaces, this implies that 
\begin{equation*}
\L_2\colon j_\phi(\M)\to L_2(\M,\phi),\,\L_2(j_\phi(x))=j_\phi(\L(x))
\end{equation*}
is bounded with the same norm. Since $\L$ is GNS-symmetric, it follows that
\begin{equation*}
\L_2(x\Omega_\phi)=\L_2(j_\phi(\sigma^\phi_{i/4}(x)\Omega_\phi))=j_\phi(\L(\sigma^\phi_{i/4}(x)))=\Delta_\phi^{1/4}(\sigma^\phi_{i/4}(\L(x)))\Omega_\phi=\L(x)\Omega_\phi
\end{equation*} 
for all $x\in \M_{\sigma^\phi}$, and thus by density for all $x\in \M$. Therefore
\begin{equation*}
\phi(\L(x)^\ast x)=\langle \L_2(x\Omega_\phi),x\Omega_\phi\rangle_2\leq \norm{\L_2}\norm{x\Omega_\phi}_2^2=\norm{\L}\phi(x^\ast x).\qedhere
\end{equation*}
\end{proof}

\begin{lemma}\label{lem:all_invariant}
There exists $\xi\in F$ such $V_t\xi=\xi$ for all $t\in \IR$, $\J(\xi\otimes\Omega_\phi)=\xi\otimes\Omega_\phi$ and $\delta_{i\xi}=\xi$.
\end{lemma}
\begin{proof}
By \Cref{lem:V_t_invariant} there exists $\xi'\in F$ such that $V_t\xi'=\xi'$ for all $t\in\IR$ and $\delta_{\xi'}=\delta$. Let $\H=F\bar\odot L_2(\M,\phi)$ and $\eta=\xi'\otimes\Omega_\phi\in\H$. By \Cref{lem:ident_left_bounded} we have $\eta\in L_\infty(\H_\M,\phi)$. We will show that $\eta\in L_\infty(_\M\H,\phi)$ as well.

For $x\in \M_{\sigma^\phi}$ we have
\begin{align*}
x\eta=x\xi'\otimes\Omega_\phi=\delta(x)\otimes\Omega_\phi+\xi' x\otimes\Omega_\phi=\delta(x)\otimes\Omega_\phi+\eta \sigma^\phi_{i/2}(x)
\end{align*}
and hence
\begin{align*}
\norm{x\eta}_\H&\leq \norm{\eta \sigma^\phi_{i/2}(x)}_\H+\norm{\delta(x)\otimes\Omega_\phi}_\H\\
&\leq \norm{L(\eta)}\norm{\Omega_\phi\sigma^\phi_{i/2}(x)}_2+\langle \Omega_\phi,(\delta(x)|\delta(x))\Omega_\phi\rangle_2\\
&=\norm{\xi'}\norm{x\Omega_\phi}_2+\phi(\L(x)^\ast x)^{1/2}\\
&\leq (\norm{\xi'}+\norm{\L})\norm{x\Omega_\phi}_2
\end{align*}
by \Cref{lem:generator_bounded}. Thus $\eta\in L_\infty(_\M\H,\phi)$ and therefore $\J\eta\in L_\infty(\H_\M,\phi)$ (see \Cref{rmk:J_left_bounded}).

By \Cref{lem:ident_left_bounded} there exists $\zeta\in F$ such that $\J\eta=\zeta\otimes\Omega_\phi$. We have
\begin{equation*}
V_t\zeta\otimes\Omega_\phi=U_t\J\eta=\J( V_t\xi'\otimes\Omega_\phi)=\J(\xi'\otimes\Omega_\phi)=\J\eta=\zeta\otimes\Omega_\phi,
\end{equation*}
where we used $\J U_t=U_t \J$ from \Cref{prop:Tomita_bimodule} (e). Hence $V_t\zeta=\zeta$.

Thus $\xi=\frac 1 {2i}(\xi'-\zeta)$ satisfies $V_t\xi=\xi$ for all $t\in \IR$,
\begin{equation*}
\J(\xi\otimes\Omega_\phi)=\frac i{2}(\J(\xi'\otimes\Omega_\phi)-\J(\zeta\otimes\Omega_\phi))=\frac i{2}(\zeta\otimes\Omega_\phi-\xi'\otimes\Omega_\phi)=\J(\xi\otimes\Omega_\phi)
\end{equation*}
and
\begin{align*}
i(x\xi-\xi x)\otimes\Omega_\phi&=\frac 1 2\delta(x)\otimes\Omega_\phi-\frac 1 2(x\zeta -\zeta x)\otimes\Omega_\phi\\
&=\frac 1 2\delta(x)\otimes\Omega_\phi-\frac 1 2(x\J\eta-(\J\eta)\sigma^\phi_{i/2}(x))\\
&=\frac 1 2\delta(x)\otimes\Omega_\phi-\frac 1 2\J(\eta x^\ast-\sigma^\phi_{i/2}(x)^\ast \eta)\\
&=\frac1  2\delta(x)\otimes\Omega_\phi+\frac 1 2\J((\sigma^\phi_{i/2}(x)^\ast \xi'-\xi' x^\ast)\otimes\Omega_\phi)\\
&=\frac 1 2\delta(x)\otimes\Omega_\phi+\frac1  2\J(\delta(\sigma^\phi_{i/2}(x)^\ast)\otimes\Omega_\phi)\\
&=\delta(x)\otimes\Omega_\phi
\end{align*}
for $x\in \M_{\sigma^\phi}$. Therefore $\delta_{i\xi}=\delta$.
\end{proof}

\begin{remark}
By \Cref{prop:Tomita_bimodule} (b), $\J(\xi\otimes\Omega_\phi)=\xi\otimes\Omega_\phi$ implies
\begin{equation*}
\langle a(\xi\otimes\Omega_\phi)x,b(\xi\otimes \Omega_\phi)y\rangle=\langle \J(b(\xi\otimes \Omega_\phi)y),\J(a(\xi\otimes \Omega_\phi)x)\rangle=\langle y^\ast(\xi\otimes\Omega_\phi)b^\ast,x^\ast(\xi\otimes\Omega_\phi)a^\ast\rangle.
\end{equation*}
\end{remark}

We have now gathered all the necessary ingredients to prove the main result of this section. However, before doing so, we want to show how the von Neumann bimodule $F$ together with the extra structure described in \Cref{prop:Tomita_bimodule} can be lifted to a von Neumann algebra containing $\M$.

This construction is a special case of a result from \cite{JRS}, building on the construction of operator-valued semicircular families in \cite{Shl99}. Since the manuscript \cite{JRS} is unpublished, we give a full proof here. Additionally, since we only deal with uniformly continuous semigroups, we can avoid the ultraproduct construction from \cite{JRS}. Otherwise, the proof follows the same lines as the one in \cite{JRS}, with some minor modifications.

One first observations is that the maps $(V_t)$ and $\J$ take similar roles as the modular automorphism group and modular conjugation on a von Neumann algebra. In fact, this is more than just an analogy, as the following theorem shows. Note that if $\N_1\subset \N_2$ is a von Neumann subalgebra and $\psi$ a normal faithful state on $\N_1$ such that $\sigma^\psi$ leaves $\N_1$ invariant, by Takesaki's theorem \cite[Theorem IX.4.2]{Tak03} there exists a unique $\psi$-preserving conditional expectation $E$ from $\N_2$ to $\N_1$, which is normal. Endowed with the $\N_1$-valued inner product
\begin{equation*}
(a|b)=E(a^\ast b)
\end{equation*}
this makes $\N_2$ into a von Neumann $\N_1$-bimodule and
\begin{equation*}
\N_2\otimes L_2(\N_1,\psi|_{\N_1})\to L_2(\N_2,\psi),\,x\otimes \eta\mapsto x\eta
\end{equation*}
extends to a unitary from $\N_2\bar\odot L_2(\N_1,\psi|_{\N_1})$ to $L_2(\N_2,\psi)$. In the statement of the following theorem we will identify these two spaces via this map.

Moreover, recall that the centralizer of a normal faithful state $\psi$ on a von Neumann algebra $\N$ is the set of all $x\in \N$ such that $\sigma^\psi_t(x)=x$ for all $t\in \IR$, or, equivalently, $\psi(xy)=\psi(yx)$ for all $y\in \N$.

\begin{theorem}\label{thm:Fock_space}
There exists a von Neumann algebra $\hat \M$ containing $\M$ as a von Neumann subalgebra with a normal faithful conditional expectation $E\colon\hat\M\to\M$, a self-adjoint element $a$ of the centralizer of $\phi\circ E$ with $E(a)=0$ and a bimodule map $\alpha\colon F\to\hat \M$ such that
\begin{enumerate}[(a)]
\item $E(\alpha(\zeta_1)^\ast \alpha(\zeta_2))=(\zeta_1|\zeta_2)$ for all $\zeta_1,\zeta_2\in F$,
\item $\alpha(V_t \zeta)=\sigma^{\phi\circ E}_t(\alpha(\zeta))$ for all $\zeta\in F$, $t\in \IR$,
\item $(\alpha\otimes \id)(\J\eta)=J_{\phi\circ E}(\alpha\otimes\id)(\eta)$ for all $\eta\in F\bar\odot L_2(\M,\phi)$,
\item $\alpha(\delta(x))=i[x,a]$ for all $x\in \M$.
\end{enumerate}
\end{theorem}
\begin{proof}
Let $\H=F\bar\odot L_2(\M,\phi)$ and 
\begin{equation*}
\mathcal F_\M(\H)=L_2(\M,\phi)\oplus\bigoplus_{k=1}^\infty \H^{\otimes_\phi k}.
\end{equation*}
This is an $\M$-$\M$-correspondence with the obvious actions.

By \Cref{lem:all_invariant} there exists $\xi\in F$ such that $V_t\xi=\xi$ for all $t\in \IR$, $\J(\xi\otimes\Omega_\phi)=\xi\otimes\Omega_\phi$ and $\delta_{i\xi}=\delta$. Let $\hat\eta=\xi\otimes\Omega_\phi$. By \Cref{lem:ident_left_bounded} we have $\hat\eta\in L_\infty(\H_\M,\phi)$ and since $\J\hat\eta=\hat\eta$ also $\hat\eta\in L_\infty(_\M\H,\phi)$. Define
\begin{align*}
a(\hat\eta)\Omega_\phi x&=\check \eta x\\
a(\check \eta)\eta&=\hat\eta\otimes_\phi \eta
\end{align*}
for $x\in \M$ and $\eta\in \H^{\otimes_\phi k}$. Since $\hat\eta\in L_\infty(\H_\M,\phi)$, this map extends to a bounded linear operator on $\mathcal F_\M(\H)$, still denoted by $a(\hat\eta)$. Let $s(\hat\eta)=a(\hat\eta)+a(\hat\eta)^\ast$.

Let $\hat \M$ be the von Neumann algebra generated by $s(\hat\eta)$ and the left action of $\M$ on $\mathcal F_\M(\H)$. Since $s(\check \eta)$ is a right module map and $\M$ acts on $\mathcal F_\M(\H)$ from the left by right module maps, the von Neumann algebra $\hat \M$ consists of right module maps. In particular, $b\Omega_\phi\in L_\infty(\mathcal F_\M(\H)_\M,\phi)$ for every $b\in\hat \M$. Thus
\begin{equation*}
E\colon \hat\M\to B(L_2(\M,\phi)),\,b\mapsto L(\Omega_\phi)^\ast b L(\Omega_\phi)=L(\Omega_\phi)^\ast L(b\Omega_\phi)
\end{equation*}
takes values in $\M$. Clearly $E$ is completely positive, so that it is a conditional expectation. Let $\hat\phi=\phi\circ E=\langle\Omega_\phi,\cdot\;\Omega_\phi\rangle_{\mathcal F_\M(\H)}$.

We will show that $\Omega_\phi$ is cyclic and separating for $\hat \M$. First note that $L_2(\M_\phi)=\overline{\M\Omega_\phi}\subset\overline{\hat\M \Omega_\phi}$. Moreover, the linear hull of $\{x\delta(y)\mid x,y\in \M\}$ is by definition strongly dense in $F$, so that the linear hull of $\{x\delta(y)\otimes\Omega_\phi\mid x,y\in \M\}$ is dense in $\H$. Since
\begin{equation*}
x\delta(y)\otimes \Omega_\phi=i(xy\xi-x\xi y)\otimes\Omega_\phi=i(xy a(\check \eta)-x a(\hat\eta)y)\Omega_\phi,
\end{equation*}
it follows that $\H\subset \overline{\hat\M\Omega_\phi}$.

From here on one can proceed by induction. Assume that $\H^{\otimes_\phi k}\subset \overline{\hat \M\Omega_\phi}$. For $x,y\in \M$ and $\eta\in \H^{\otimes_\phi k}$ we have
\begin{equation*}
(x\delta(y)\otimes\Omega_\phi)\otimes_\phi \eta=i(xy a(\hat\eta)-x a(\hat\eta)y)\eta\in \overline{\hat \M\Omega_\phi}.
\end{equation*}
By \Cref{lem:ident_left_bounded} the linear hull of $\{x\delta(y)\otimes\Omega_\phi\mid x,y\in\M\}$ is strongly dense in $L_\infty(\H_\M,\phi)$. By the definition of the relative tensor product, this implies $\H\otimes_\phi\H^{\otimes_\phi k}\subset \overline{\hat\M\Omega_\phi}$.

Finally, as elements with finitely many non-zero entries in the direct sum decomposition are dense in $\mathcal F_\M(\H)$, we conclude $\mathcal F_\M(\H)\subset \overline{\hat \M\Omega_\phi}$. In other words, $\Omega_\phi$ is cyclic for $\hat \M$.

To show that $\Omega_\phi$ is separating for $\hat \M$, we prove that it is cyclic for $\hat \M'$. For that purpose define 
\begin{align*}
b(\hat\eta)x\Omega_\phi&=x\hat\eta\\
b(\hat\eta)\eta&=\eta\otimes_\phi\hat\eta
\end{align*}
for $x\in \M$ and $\eta\in\H^{\otimes\phi k}$. Since $\J\hat\eta=\hat\eta$ and $\hat\eta$ is left-bounded, it is also right bounded and $b(\hat\eta)$ extends to a bounded linear operator on $\mathcal F_\M(\H)$, still denoted by $b(\hat\eta)$. Let $t(\hat\eta)=b(\hat\eta)+b(\hat\eta)^\ast$.

Clearly, $b(\hat\eta)$ commutes with the left action of $\M$ on $\mathcal F_\M(\H)$. We will next show that $s(\hat\eta)$ commutes with $t(\hat\eta)$. First note that
\begin{align*}
\langle b(\hat\eta)^\ast \eta,x\Omega_\phi\rangle_2=\langle \eta,x\hat\eta\rangle_\H=\langle \hat\eta x^\ast,\J\eta\rangle_\H=\langle \Omega_\phi x^\ast,a(\hat\eta)^\ast\J\eta\rangle_\H=\langle J a(\hat\eta)^\ast \J \eta,x\Omega_\phi\rangle_2
\end{align*}
for $x\in \M$ and $\eta\in \H$ and of course $a(\check \eta)^\ast|_{\H}=L(\hat\eta)^\ast$.

For $x\in\M_{\sigma^\phi}$ we have
\begin{align*}
t(\hat\eta)s(\hat\eta)\Omega_\phi x&=t(\hat\eta)\hat\eta x\\
&=\hat\eta x\otimes_\phi\hat\eta+J a(\hat\eta)^\ast\J(\hat\eta x)\\
&=\hat\eta\otimes_\phi \sigma^\phi_{-i/2}(x)\hat\eta+J L(\hat\eta)^\ast  L(x^\ast\hat\eta)J\Omega_\phi\\
&=a(\hat\eta)t(\hat\eta)\sigma^\phi_{-i/2}(x)\Omega_\phi+\Omega_\phi (x^\ast\xi|\xi)\\
&=a(\hat\eta)t(\hat\eta)\Omega_\phi x+\sigma^\phi_{-i/2}((\xi|x\xi))\Omega_\phi\\
&=a(\hat\eta)t(\hat\eta)\Omega_\phi x+(\xi|\sigma^\phi_{-i/2}(x)\xi)\Omega_\phi\\
&=a(\hat\eta)t(\hat\eta)\Omega_\phi x+L(\hat\eta)^\ast \sigma^\phi_{-i/2}(x)\hat\eta\\
&=a(\hat\eta)t(\hat\eta)\Omega_\phi x+a(\hat\eta)^\ast t(\hat\eta)\Omega_\phi x,
\end{align*}
where we used \Cref{lem:mod_group_inner_prod}. The verification for commutation on $\H^{\otimes_\phi k}$ is similar. Therefore, $\hat\M^\prime$ contains $t(\hat\eta)$ and the right action of $\M$ on $\mathcal F_\M(\H)$. Hence $\Omega_\phi$ is cyclic for $\hat\M'$ by an analog proof as for $\hat \M$.

Therefore $E$ and consequently also $\hat\phi$ is faithful. If $b\in\hat \M$, then
\begin{equation*}
\langle \Omega_\phi,s(\hat\eta)b\Omega_\phi\rangle=\langle \hat\eta,b\Omega_\phi\rangle=\langle t(\hat\eta)\Omega_\phi,b\Omega\phi\rangle=\langle \Omega_\phi,t(\hat\eta)b\Omega_\phi\rangle=\langle \Omega_\phi,b t(\hat\eta)\Omega_\phi=\langle\Omega_\phi,b s(\hat\eta)\Omega_\phi\rangle
\end{equation*}
since $t(\hat\eta)\in\hat\M^\prime$. Thus $s(\hat\eta)$ is in the centralizer of $\hat\phi$.

Define a map $\alpha$ from the linear hull of $\{x\xi y\mid x,y\in \M\}\subset F$ to $\hat\M$ by $\alpha(x\xi y)=x s(\hat\eta) y$. Then
\begin{equation*}
\alpha(x\xi y)\Omega_\phi=x s(\hat\eta)y\Omega_\phi=x\xi y\otimes \Omega_\phi.
\end{equation*}
Thus $\alpha$ is well-defined and strongly continuous. It is easy to check that the extension to $F$ satisfies properties (a)--(d).
\end{proof}

\begin{corollary}
Let $\M$ be a von Neumann algebra and $\phi$ a normal faithful state on $\M$. For every uniformly continuous $\phi$-symmetric quantum Markov semigroup $(P_t)$ with carré du champ $\Gamma$ there exists a von Neumann algebra $\hat\M\supset \M$ with a normal faithful conditional expectation $E\colon \hat\M\to \M$ and a self-adjoint element $a$ in the centralizer of $\phi\circ E$ with $E(a)=0$ such that
\begin{equation*}
\Gamma(x,y)=E([x,a]^\ast [y,a]).
\end{equation*}
\end{corollary}

Now we come to the announced refinement of the Christensen--Evans theorem for generators of uniformly continuous GNS-symmetric quantum Markov semigroups.

\begin{theorem}\label{thm:CE_symmetric}
Let $\M$ be a von Neumann algebra and $\phi$ a normal faithful state on $\M$. If $(P_t)$ is a uniformly continuous $\phi$-symmetric quantum Markov semigroup with generator $\L$, then there exists a $\phi$-symmetric completely positive map $\Phi\colon\M\to\M$ such that
\begin{equation*}
\L(x)=\frac 1 2(\Phi(1)x+x\Phi(1))-\Phi(x)
\end{equation*}
for all $x\in \M$.

Conversely, every operator of this form generates a uniformly continuous $\phi$-symmetric quantum Markov semigroup.
\end{theorem}
\begin{proof}
First assume that we are given a uniformly continuous $\phi$-symmetric quantum Markov semigroup with generator $\L$. We will use the von Neumann $\M$-bimodule with the extra structure constructed in this section. By \Cref{lem:all_invariant} there exists $\xi\in F$ such that $V_t\xi=\xi$, $\J(\xi\otimes\Omega_\phi)=\xi\otimes\Omega_\phi$ and $\delta_{i\xi}=\delta$. By definition of $\delta$ this means
\begin{align*}
\phi(\L(x)^\ast y)&=\phi((\delta(x)|\delta(y)))\\
&=\phi((x\xi|y\xi)+(\xi x|\xi y)-(x\xi|\xi y)-(\xi x|y\xi))\\
&=\langle x\xi\otimes\Omega_\phi,y\xi\otimes\Omega_\phi\rangle+\phi(x^\ast (\xi|\xi)y)-\phi((\xi|x\xi)^\ast y)-\langle \xi x\otimes\Omega_\phi,y\xi\otimes\Omega_\phi\rangle\\
&=\langle \xi\otimes \Omega_\phi y^\ast,\xi\otimes\Omega_\phi x^\ast\rangle+\phi(x^\ast(\xi|\xi)y)-\phi((\xi|x\xi)^\ast y)-\langle \xi\otimes\Omega_\phi y^\ast,\sigma^\phi_{i/2}(x)^\ast\xi\otimes\Omega_\phi\rangle\\
&=\phi((\xi|\xi \sigma^\phi_{-i/2}(x^\ast y)))+\phi(x^\ast(\xi|\xi)y)-\phi((\xi|x\xi)^\ast y)-\phi((\xi|\sigma^\phi_{-i/2}(x^\ast)\xi \sigma^\phi_{-i/2}(y))\\
&=\phi(\sigma^\phi_{-i/2}((\xi|\xi)x^\ast y))+\phi(x^\ast(\xi|\xi)y)-\phi((\xi|x\xi)^\ast y)-\phi(\sigma^\phi_{-i/2}(\xi|x^\ast \xi y))\\
&=\phi((\xi|\xi)x^\ast y+x^\ast (\xi|\xi)y-2(\xi|x\xi)^\ast y)
\end{align*}
for all $x,y\in\M_{\sigma^\phi}$.

Thus
\begin{equation*}
\L(x)=(\xi|\xi)x+x(\xi|\xi)-2(\xi|x\xi).
\end{equation*}
Let $\Phi\colon \M\to\M,\,x\mapsto 2(\xi|x\xi)$. Clearly, this map is completely positive. Moreover, if $x,y\in\M_{\sigma^\phi}$, then
\begin{align*}
\phi(\Phi(x)^\ast y)&=2\phi((x\xi|\xi y))\\
&=2\langle x\xi\otimes\Omega_\phi,\xi y\otimes\Omega_\phi\rangle\\
&=2\langle \sigma^\phi_{i/2}(y)^\ast \xi\otimes\Omega_\phi,\xi\sigma^\phi_{-i/2}(x^\ast)\otimes\Omega_\phi\rangle\\
&=2\phi((\xi|\sigma^\phi_{i/2}(y)\xi\sigma^\phi_{-i/2}(x^\ast)))\\
&=2\phi(\sigma^\phi_{i/2}(x^\ast)\sigma^\phi_{i/2}((\xi|y\xi)))\\
&=\phi(x^\ast\Phi(y)).
\end{align*}
Thus $\Phi$ is $\phi$-symmetric.
\end{proof}

\begin{remark}
Since the map $x\mapsto (\xi|x\xi)$ is $\phi$-symmetric, it commutes with $\sigma^\phi$ by \Cref{prop:mod_group_commute}. Thus
\begin{equation*}
\sigma^\phi_t((\xi|\xi))=(\xi|\sigma^\phi_t(1)\xi)=(\xi|\xi)
\end{equation*}
for all $t\in\IR$, that is, $(\xi|\xi)$ is in the centralizer of $\phi$.
\end{remark}

%In the case when the action of $\M$ on $H$ is a cyclic representation, the tensor product $F\bar\odot H$ can be obtained as completion of $F$ with respect to an inner product. The proof is straightforward an will be omitted.
%\begin{lemma}
%Let $\M$ be a von Neumann algebra with normal faithful state $\phi$, $F$ a von Neumann $\M$-$\M$-bimodule and $\H$ its completion with respect to the inner product
%\begin{equation*}
%\langle \cdot,\cdot\rangle_\H\colon F\times F\to\IC,\,(\xi,\eta)\mapsto\phi((\xi|\eta)).
%\end{equation*}
%Write $\Lambda$ for the canonical map from $F$ into $H$ and $(H_\phi,\pi_\phi,\Omega_\phi)$ for the cyclic representation associated with $\phi$. The map
%\begin{equation*}
%\Lambda(F)\to F\bar\odot H_\phi,\,\Lambda(\xi)\mapsto\xi\otimes \Omega_\phi
%\end{equation*}
%extends to a unitary $U$ from $\H$ to $F\bar\odot H_\phi$ such that 
%\begin{equation*}
%U(\Lambda(x\xi))=\pi_\phi(x)U(\Lambda(\xi))
%\end{equation*}
%for all $x\in \M$, $\xi\in F$.
%\end{lemma}
%
%Under this identification, $\H$ carries a natural right action of $\M$, which is given by
%\begin{equation*}
%\Lambda(\xi)x=\Lambda(\xi\sigma^\phi_{-i/2}(x))
%\end{equation*}
%for $\xi\in F$ and $x\in \M_{\sigma^\phi}$. This makes $\H$ into an $\M$-$\M$-correspondence in the sense of Connes.

\begin{example}
The simplest class of examples are quantum Markov semigroups with generators of the form $\L=I-\Phi$ with unital completely positive $\Phi$. Evidently, such a quantum Markov semigroup is GNS-symmetric if and only $\Phi$ is. Thus the generator is already in GNS-symmetric Christensen--Evans form.

In this case, the Fock space construction in the proof of \Cref{thm:Fock_space} is essentially that from \cite{Shl99}, just that Shlyakhtenko starts with $\M\bar\otimes\M$ instead of the subspace $\{\sum_j a_j\otimes x_j\mid\sum_j a_jx_j=0\}$ in the definition of $F$, which allows to immediately write $\partial(x)=x\xi-\xi x$ with $\xi=1\otimes 1$.

A special class of quantum Markov semigroups of this form are the generalized depolarizing (or dephasing) semigroups $P_t=e^{-t}I+(1-e^{-t})E$ with a normal conditional expectation $E$, whose generator has the form $\L=I-E$.
\end{example}

\begin{example}
Let $G$ be a finite group and $\ell\colon G\to [0,\infty)$ a conditionally negative definite length function. As discussed in \Cref{ex:groups}, the operators $P_t$ given by $P_t\lambda_g=e^{-t\ell(g)}\lambda_g$ form a $\tau$-symmetric quantum Markov semigroup on the group algebra $L(G)$. Let 
\begin{equation*}
\xi=\frac{i}{\abs{G}}\sum_{g\in G}\delta(\lambda_g)\lambda_{g^{-1}}.
\end{equation*}
A similar calculation as in \Cref{thm:AFD_amenable} shows that $\delta_{i\xi}=\delta$. Moreover,
\begin{equation*}
\J(\xi\otimes\Omega_\phi)=-\frac{i}{\abs{G}}\sum_{g\in G}\lambda_g\delta(\lambda_{g^{-1}})\otimes\Omega_\phi=-\frac{i}{\abs{G}}\sum_{g\in G}(\delta(\lambda_g\lambda_{g^{-1}})-\delta(\lambda_g)\lambda_{g^{-1}})\otimes\Omega_\phi=\xi\otimes\Omega_\phi,
\end{equation*}
where we used that the modular group is trivial since $\tau$ is a trace. Furthermore, the same fact also implies that $(V_t)$ is trivial, so that $V_t\xi=\xi$ is automatic.

A direct calculation shows
\begin{equation*}
(\xi|\lambda_g \xi)=\frac 1{\abs{G}^2}\sum_{h,h'\in G}K(h,h')\lambda_g-\frac 1{\abs{G}}\sum_{h\in G}K(h,g)\lambda_g=\frac 1{\abs{G}^2}\sum_{h,h'\in G}K(h,h')\lambda_g-\frac 1 2\ell(g)\lambda_g
\end{equation*}
with $K(h,h')=\frac 1 2(\ell(h)+\ell(h')-\ell(h^{-1}h'))$.

As shown in the proof of \Cref{thm:CE_symmetric}, 
\begin{equation*}
\L(\lambda_g)=(\xi|\xi)\lambda_g+\lambda_g(\xi|\xi)-2(\xi|\lambda_g\xi).
\end{equation*}
Indeed,
\begin{equation*}
(\xi|\xi)\lambda_g+\lambda_g(\xi|\xi)-2(\xi|\lambda_g\xi)=\frac{2}{\abs{G}^2}\sum_{h,h'}K(h,h')\lambda_g-\frac 2{\abs{G}^2}\sum_{h,h'\in G}K(h,h')\lambda_g+\ell(g)\lambda_g=\ell(g)\lambda_g.
\end{equation*}
\end{example}

\begin{example}\label{ex:graphs}
Let $X$ be a finite set, $m$ a probability measure on $X$ of full support and $(P_t)$ a Markov semigroup on $C(X)$. (Completely) positive maps on $C(X)$ are of the form
\begin{equation*}
\Phi\colon C(X)\to C(X),\,\Phi f(x)=\sum_y Q(x,y)f(y)
\end{equation*}
with $Q(x,y)\geq 0$, and $\Phi$ is (GNS-) symmetric if and only if
\begin{equation*}
Q(x,y)m(x)=Q(y,x)m(y).
\end{equation*}
As a consequence of \Cref{thm:CE_symmetric}, the generators of symmetric Markov semigroups on $C(X)$ are exactly the maps of the form
\begin{equation*}
\L f(x)=\sum_{y}Q(x,y)(f(x)-f(y))
\end{equation*}
with weights $Q$ satisfying the positivity and symmetry condition mentioned above.

Thus we recover the classical fact that generators of symmetric Markov semigroups on $C(X)$ are exactly the graph Laplacians of undirected weighted graphs over $X$, or, equivalently, the generators of reversible Markov processes om $X$.
\end{example}

\section{Extensions of quantum Markov semigroups}

In this section we use our version of the Christensen--Evans theorem for GNS-symmetric quantum Markov semigroups to show that these can be extended from subalgebras to the whole algebra under natural conditions (\Cref{thm:extension}). In particular, this implies that Alicki's theorem is applicable to generators of GNS-symmetric quantum Markov semigroups on arbitrary finite-dimensional von Neumann algebras (\Cref{cor:Alicki}), which answers a question raised in \cite{CM17a}.

\begin{theorem}[Extensions of GNS-symmetric quantum Markov semigroups]\label{thm:extension}
Let $\M\subset \hat\M$ be an inclusion of von Neumann algebras and assume that there exists a normal conditional expectation $E$ from $\hat\M$ onto $\M$. Every uniformly continuous quantum Markov semigroup $(P_t)$ on $\M$ can be extended to a uniformly continuous quantum Markov semigroup $(\hat P_t)$ on $\hat \M$.

Moreover, if $\hat \phi$ is a normal faithful state on $\hat \M$ such that the modular group $\sigma^{\hat\phi}$ leaves $\M$ invariant, $\phi$ is the restriction of $\hat\phi$ to $\M$ and $(P_t)$ is $\phi$-symmetric, then $(\hat P_t)$ can be chosen $\hat \phi$-symmetric.
\end{theorem}
\begin{proof}
By the Christensen--Evans theorem \cite[Theorem 3.1]{CE79}, there exists $k\in \M$ and a completely positive map $\Phi\colon \M\to \M$ with $\Phi(1)=k+k^\ast$ such that the generator $\L$ of $(P_t)$ satisfies
\begin{equation*}
\L(x)=k^\ast x+x k-\Phi(x)
\end{equation*}
for all $x\in \M$. Let
\begin{equation*}
\hat\L\colon\hat\M\to\hat\M,\,\hat \L(x)=k^\ast x+x k-\Phi(E(x)).
\end{equation*}
Clearly $\hat\L$ is an extension of $\L$ and again of Christensen--Evans form, hence the generator of a quantum Markov semigroup on $\hat \M$.

In the situation of the second paragraph, $k$ can be chosen in the centralizer of $\sigma^{\phi}$ and $\Phi$ can be chosen $\phi$-symmetric by \Cref{thm:CE_symmetric}. Moreover, $E$ can be chosen $\phi$-symmetric as well by Takesaki's theorem \cite[Theorem IX.4.2]{Tak03}. Then the operator $\hat \L$ defined above generates a uniformly continuous $\phi$-symmetric quantum Markov semigroup.
\end{proof}

\begin{remark}
The extension result for quantum Markov semigroups without symmetry assumptions is an immediate consequence of the original Christensen--Evans theorem and may have been known before, although we could not find it in the literature. In contrast, the existence of an extension preserving GNS-symmetry was raised as a question in \cite[Remark A.12]{CM17a}. This construction of GNS-symmetric extensions might also prove useful to verify the geometric Ricci condition from \cite[Section 3]{LJL20}, \cite[Definition 3.23]{BGJ22}.
\end{remark}

\begin{example}
Let $m$ be a probability measure of full support on $\{1,\dots,n\}$ and let $(P_t)$ be a symmetric Markov semigroup on $\ell_\infty(\{1,\dots,n\})$. By \Cref{ex:graphs} the generator $\L$ of $(P_t)$ is of the form
\begin{equation*}
\L f(j)=\sum_j Q(j,k)(f(j)-f(k))
\end{equation*}
with $Q\colon \{1,\dots,n\}\times \{1,\dots,n\}\to [0,\infty)$ and $Q(j,k)m(j)=Q(k,j)m(k)$.

View $\ell_\infty(\{1,\dots,n\})$ as algebra of diagonal matrices inside $M_n(\IC)$ and define a faithful state on $M_n(\IC)$ by $\phi(A)=\sum_{j=1}^n A_{jj}m(j)$. Let $B=\sum_{j,k=1}^n Q(j,k)E_{jj}$, where $E_{jl}$ denotes the usual matrix units. Then the extension of $(P_t)$ to a $\phi$-symmetric quantum Markov semigroup on $M_n(\IC)$ provided by \Cref{thm:extension} has generator $\hat\L$ given by
\begin{equation*}
\hat \L(A)=\sum_{j,k=1}^n Q(j,k)\left(\frac{E_{jj}A+AE_{jj}}2-A_{kk}E_{jj}\right).
\end{equation*}
\end{example}

Now we come to the announced application to Alicki's theorem. A density matrix is a positive element $\sigma\in M_n(\IC)$ with trace $1$. If $\M\subset M_n(\IC)$ is a unital $\ast$-subalgebra, every state on $\M$ is of the form $\phi=\mathrm{tr}(\,\cdot\,\sigma)$ for some density matrix $\sigma\in \M$. The state $\phi$ is faithful if and only if $\sigma$ has full rank.

Let $\sigma\in\M$ be a full-rank density matrix. A QMS $(P_t)$ on $\M$ is said to satisfy the $\sigma$-detailed balance condition (DBC) if it is GNS-symmetric with respect to the state $\mathrm{tr}(\,\cdot\,\sigma)$.

\begin{corollary}\label{cor:Alicki}
Let $\M\subset M_n(\IC)$ be a unital $\ast$-subalgebra and $\sigma\in \M$ a full-rank density matrix. If $(P_t)$ is a quantum Markov semigroup on $\M$ satisfying the $\sigma$-DBC, then it can be extended to a quantum Markov semigroup $(\hat P_t)$ on $M_n(\IC)$ satisfying the $\sigma$-DBC.

In particular, there exists a finite set $\mathcal J$, real numbers $\omega_j$, positive numbers $c_j$ and elements $v_j\in M_n(\IC)$ for $j\in \mathcal J$ with
\begin{enumerate}[(i)]
\item $\mathrm{tr}(v_j)=0$ for all $j\in\mathcal J$,
\item $\mathrm{tr}(v_j^\ast v_k)=\delta_{jk}$ for all $j,k\in \mathcal J$,
\item $\{v_j\mid j\in\mathcal J\}=\{v_j^\ast\mid j\in \mathcal J\}$,
\item $\sigma v_j \sigma^{-1}=e^{-\omega_j}v_j$ for all $j\in\mathcal J$
\end{enumerate}
such that the generator $\L$ of $(P_t)$ satisfies
\begin{equation*}
\L=\sum_{j\in\mathcal J} c_j\left(e^{-\omega_j/2}v_j^\ast[v_j,\cdot\,]-e^{\omega_j/2}[v_j^\ast,\cdot\,]v_j\right).
\end{equation*}
\end{corollary}
\begin{proof}
The extension part is a special case of \Cref{thm:extension}, where the assumption on the existence of a conditional expectation is always satisfied in finite dimensions. The representation of the generator follows then from Alicki's theorem \cite[Theorem 3]{Ali76}, \cite[Theorem 3.1]{CM17a}.
\end{proof}

\appendix
\section{AFD von Neumann algebras are amenable}

The hard part of the proof of the Christensen--Evans theorem is to show that every bounded derivation on a von Neumann algebra with values in a von Neumann bimodule is inner \cite[Theorem 2.1]{CE79}, and we also use this result in our proof of \Cref{thm:CE_symmetric}. In the case when the von Neumann algebra is finite-dimensional, which is of special interest to us for the application to Alicki's theorem, we can give a considerably simpler proof of this result.

In fact, this proof covers not only finite-dimensional von Neumann algebras, but more generally approximately finite-dimensional ones, and not only derivations with values in a von Neumann bimodule, but more generally derivations with values in a dual Banach bimodule. Let us summarize the relevant terminology.

If $A$ is a Banach algebra, a \emph{Banach $A$-bimodule} is an $A$-bimodule $F$ with a Banach norm such that
\begin{equation*}
\norm{a\xi b}\leq \norm{a}\norm{\xi}\norm{b}
\end{equation*}
for all $a,b\in A$ and $\xi\in F$.

If $F$ is the dual of a Banach space $F_\ast$ and for every $a\in A$ the maps $\xi\mapsto a\xi$ and $\xi\mapsto\xi a$ are weak$^\ast$ continuous, then $F$ is called a \emph{dual} Banach $A$-bimodule. By \cite[Proposition 3.8]{Pas73}, \cite[Theorem 3.2.11]{Ske01}, every von Neumann bimodule is a dual Banach bimodule.

If $\M$ is a von Neumann algebra and $F$ a dual Banach $\M$-bimodule, a bounded derivation $\delta\colon \M\to F$ is called \emph{normal} if for every $\omega\in F_\ast$ the map $a\mapsto \omega(\delta(a))$ is weak$^\ast$ continuous.

A von Neumann algebra $\M$ is called \emph{amenable} if every derivation on $\M$ with values in a dual Banach $\M$-bimodule is inner. A von Neumann algebra with separable predual is called \emph{approximately finite-dimensional} if it is the weak closure of an increasing sequence of finite-dimensional von Neumann algebras.

The following result is originally due to Kadison and Ringrose \cite{KR71a,KR71b}, but the proof given here seems simpler and only uses approximate finite-dimensionality directly.
\begin{theorem}\label{thm:AFD_amenable}
Every approximately finite-dimensional von Neumann algebra with separable predual is ame\-nable.
\end{theorem}
\begin{proof}
Let $\M$ be an approximately finite-dimensional von Neumann algebra and $(\M_n)$ an increasing sequence of finite-dimensional von Neumann subalgebras of $\M$ whose union is weakly dense in $\M$. Let $\mathcal U(\M)$ and $\mathcal U(\M_n)$ denote be the group of unitary elements of $\M$ and $\M_n$, respectively

Let $F$ be a dual Banach $\M$-bimodule and $\delta\colon\M\to F$ a normal derivation. As $\delta$ is bounded, the sequence $(\xi_n)$ with
\begin{equation*}
\xi_n=\int_{\mathcal U(\mathcal M_n)}u^\ast\delta(u)\,du
\end{equation*}
is bounded in $F$. Here the integration is with respect to the normalized Haar measure on $\mathcal U(\M_n)$. By the Banach-Alaoglu theorem there exists a weak$^\ast$ accumulation point $\xi\in F$.

If $v\in \mathcal U(\mathcal M_n)$ and $m\geq n$, then 
\begin{align*}
v\xi_m-\xi_m v&=\int_{\mathcal U(\M_m)}(vu^\ast\delta(u)-u^\ast \delta(u)v)\,du\\
&=\int_{\mathcal U(\M_m)}(vu^\ast\delta(u)+u^\ast u\delta(v)-u^\ast\delta(uv))\,du\\
&=\delta(v)+\int_{\mathcal U(\M_m)}vu^\ast\delta(u)\,du+\int_{\mathcal U(\M_m)}u^\ast \delta(uv)\,du\\
&=\delta(v),
\end{align*}
where we used the substitution $u\mapsto uv^\ast$ for the third summand in the last step. Hence $v\xi-\xi v=\delta(v)$.

By \cite[Theorem II.4.11]{Tak02}, $\mathcal U(\M)$ is the strong$^\ast$ closure of $\bigcup_n \mathcal U(\M_n)$. As $\delta$ is normal, this implies the identity $\delta(v)=v\xi-\xi v$ for all $v\in\mathcal U(\mathcal M)$. Finally, as $\mathcal U(\M)$ spans $\M$, the same equality holds for arbitrary $x\in \M$ instead of $v$.
\end{proof}
\begin{remark}
Of course, if $\M$ itself is finite-dimensional, there is no need for approximation and one can directly define $\xi=\int_{\mathcal U(\M)}u^\ast\delta(u)\,du$.
\end{remark}

\printbibliography
\end{document}